\documentclass[12pt]{amsart}
\usepackage{ifpdf}
\usepackage[colorlinks,linkcolor=blue,citecolor=blue,urlcolor=blue, pdfcenterwindow, pdfstartview={XYZ null null 1.2}, pdffitwindow, pdfdisplaydoctitle=true]{hyperref}
\usepackage[english]{babel}
\usepackage{array}
\usepackage{threeparttable}
\usepackage{rotating}
\usepackage[section]{placeins}
\usepackage{amssymb} 
\usepackage{amsmath}
\usepackage{mathrsfs}
\usepackage{breqn}
\IfFileExists{bbm.sty}{\usepackage{bbm}}{\newcommand{\mathbbm}[1]{\mathbf{##1}}}
\usepackage{enumerate}
\usepackage{enumitem}
\usepackage{color}
\usepackage{relsize}
\newtheorem{lemma}{Lemma}[section]

\newtheorem{prop}[lemma]{Proposition}
\newtheorem{coro}[lemma]{Corollary}
\newtheorem{theo}[lemma]{Theorem}

\newtheorem{rema}[lemma]{Remark}

\newtheorem{defi}[lemma]{Definition}

\newtheorem{exam}[lemma]{Example}

\begin{document}

\title{Asymptotic Schur orthogonality for lattices}

\author{A. Bendikov}
\address{Institute of Mathematics, Wroclaw University}
\email{bendikov@math.uni.wroc.pl}
\author{A. Boyer}
\address{Universit\'e Paris Cit\'e}
\email{adrien.boyer@imj-prg.fr }
\author{Ch. Pittet}
\address{I2M CNRS UMR7373 Aix-Marseille Univ.}
\email{pittet@math.cnrs.fr}
\thanks{This work was initiated years ago (January 2018) with the help of the program ``Research in Pairs'' of the Mathematisches Forschungsinstitut Oberwolfach. The authors thank the MFO and the CIRM for
RIP contracts providing excellent research environments}



\date{September 28th, 2026}

%\tableofcontents
%
\begin{abstract} Let $\pi:G\to \mathcal{U}(L^{2}(G/P,\nu))$ be the boundary representation of a non-compact connected semisimple Lie group $G$ with finite center on its Furstenberg-Poisson boundary $(G/P,\nu)$. Let $\Gamma\subset G$ be a lattice in $G$ (uniform or not). We show that for any \emph{continuous} functions $\varphi,\psi,\varphi',\psi'$ on $G/P$,
$$
\lim_{n\to\infty}\frac{1}{|\Gamma_{n}|}\sum_{\gamma\in \Gamma_{n}}\frac{\langle\pi(\gamma)\varphi,\psi\rangle\overline{\langle\pi(\gamma)\varphi',\psi'\rangle}}{\Xi^{2}(\gamma)}=\langle \varphi,\varphi'\rangle\overline{\langle\psi,\psi'\rangle},
$$
where $\Gamma_{n}$ is a ball in $\Gamma$ with radius $n$ relative to a natural length function and with center  the identity element of $G$, and $\Xi$ is the restriction to $\Gamma$ of the Harish-Chandra function of $G$. As a corollary, we  deduce that when the real rank of $G$ is one, then the analogous convergence holds for any  $\varphi,\psi,\varphi',\psi'$ in $L^{2}(G/P,\nu)$ if and only if the lattice $\Gamma$ is uniform.
\end{abstract} 

\maketitle
\tableofcontents

\section*{Introduction}
Let $G$ be a finite group and let $\pi:G\to\mathcal U(H)$ be an irreducible unitary representation on a complex Hilbert space of dimension $d$. Schur's classical orthogonality relations imply that for all $v,w,v',w'\in H$, 
\[
\frac{d}{|G|}\sum_{g\in G}
\langle\pi(g)v,w\rangle
\overline{\langle\pi(g)v',w'\rangle}
=
\langle v,v'\rangle\overline{\langle w,w'\rangle}.
\]
For infinite groups and infinite-dimensional representations, one is led to replace the finite sum by suitably normalized averages over an exhausting family of subsets of the group and to ask whether the same inner-product identity emerges in the limit.

 Such asymptotic relations have been established for boundary representations of hyperbolic groups \cite{BG} and for other families of spherical boundary representations \cite{BP}. For semisimple groups over local fields, a general conjecture predicts asymptotic Schur orthogonality for every irreducible tempered unitary representation \cite{KYD}. The relations are known for $K$-finite vectors in both the non-Archimedean \cite{KYD} and Archimedean \cite{ALR} settings.

An operator-theoretic perspective on asymptotic Schur orthogonality relations is developed
in \cite{BBP} with the help of Hilbert tensor products. It allows to consider gradually stronger forms of asymptotic Schur orthogonality: convergence on a dense subspace, convergence on the whole Hilbert space, convergence in Hilbert tensor products, see  \cite[Proposition 1.1]{BBP}. In the case $G$ is semisimple over a local field, and  $\pi$ is the quasi-regular representation of $G$ on  its
Furstenberg--Poisson boundary $(B,\nu)$, it is proven that suitably normalized averages of operators of the type
$\pi(g)\otimes\pi(g^{-1})\in U(L^2(B,\nu)\otimes L^2(B,\nu))$, where $g\in G$, converge in the weak operator topology
to the flip operator $F$ characterized by the equalities $F(v\otimes w)=w\otimes v$, where $v\otimes w\in L^2(B,\nu)\otimes L^2(B,\nu)$. It implies asymptotic Schur orthogonality relations on the whole Hilbert space $L^2(B,\nu)$. It naturally raises the corresponding questions for lattices in $G$.

From now on, let $G$ be a noncompact connected semisimple Lie group
with finite center.
In the relations considered here, the coefficients are normalized locally with the help of  Harish–Chandra’s \(\Xi\)-function associated with \(G\), see Section \ref{section: behavior near infinity} below. In Theorem \ref{theo: main} below, we establish asymptotic Schur orthogonality relations on the space of continuous functions on the Furstenberg-Poisson boundary for every lattice in $G$. The proof relies on the combination of a double equidistribution result on the direct product of the Furstenberg Poisson boundary with itself (see Subsection  \ref{section: double} below) with a uniform approximation of normalized matrix coefficients, involving a Fatou type theorem (see Section  \ref{section:  behavior near infinity} below).
 A central feature of our results is that the space of vectors matters: in real rank one, the asymptotic Schur orthogonality relations extend to all square-integrable  functions on the boundary if and only if the lattice is cocompact.  Thus the domain of validity of asymptotic Schur orthogonality relations for a lattice corresponds to a geometric feature of the lattice. More precisely, suppose $G=KAN$ is an Iwasawa decomposition and $P=MAN$ is a minimal parabolic so that $G/P=K/M$. Let $\nu$ be  the Borel $K$-invariant probability measure on $G/P$ and denote $(G/P,\nu)$ the Furstenberg-Poisson boundary. Let $\Gamma<G$ be a lattice and let $\pi$ denote the quasi-regular representation on $L^2(G/P,\nu)$. Theorem~\ref{theo: main} below, as well as its more general  version expressed in Theorem~\ref{theo: main technical} below,  establish asymptotic Schur orthogonality relations  for all families of four vectors in the Banach space $C(G/P)$ of continuous functions on the Furstenberg-Poisson boundary, averaging over balls (or coronae) in $\Gamma$ with the normalization by $\Xi$ mentioned above. The results apply to uniform and nonuniform lattices alike. In rank one, Corollary~\ref{cor: equivalences in rank 1} identifies the validity of the  extension of the asymptotic Schur orthogonality relations  to the whole Hilbert space $L^2(G/P,\nu)$ of square-integrable functions on the Furstenberg-Poisson boundary, with a sharp quadratic estimate for sums of squared matrix coefficients over unit-width annuli. The same equivalence  probably  holds in any rank and  we hope to investigate this matter in a forthcoming paper.
The above mentioned quadratic estimate implies the property of rapid decay (property RD) with respect to the restriction to the lattice of the length function on $G$ induced by the distance on the symmetric space $G/K$. In rank one, property RD for this length function on the lattice  is known to hold precisely when the lattice is cocompact.

\section{Statement of the main results}
\subsection{Setting and notation}
We introduce the concepts and  notation needed to state the main result. Let $G$ be a connected semisimple Lie group with finite center and let $\frak g$ be its Lie algebra.
Let $K$ be a maximal compact subgroup of $G$ and let $\frak k$ be its Lie algebra. Let $\frak p$ be the orthogonal complement of $\frak k$ in $\frak g$
relative to the Killing form $B$. Among the  abelian sub-algebras of $\frak g$ contained in the subspace $\frak p$, let $\frak a$ be a maximal one. We assume $\dim \frak a>0$, i.e. the real rank of $G$ is strictly positive (or equivalently that $G$ is not compact). Let $\Sigma$ be the root system associated to $(\frak g,\frak a)$. Let
\[
	\frak g_{\alpha}=\{X\in\frak g: \mathrm{ad}(H)X=\alpha(H)X\ \ \forall H\in\frak a  \}
\]
be the root space of $\alpha\in\Sigma$.
Let 
\[
	\frak a^{sing}=\bigcup_{\alpha\in\Sigma}\mathrm{ker}(\alpha)
\]
be the union of the hyperplanes of $\frak a$ defined as the kernels of the roots of $\Sigma$. 
Let us choose a positive open  Weyl chamber $\frak a^+$ (i.e. a connected component of $\frak a\setminus \frak a^{sing}$).
Let $\Sigma^+\subset\Sigma$ be the set of positive roots ($\alpha\in\Sigma$ is positive if and only if $\alpha(H)>0$ for all $H\in \frak a^+$), and 
 $\frak n$ the nilpotent Lie algebra defined as the direct sum of root spaces of positive roots:
\[
	\frak n=\bigoplus_{\alpha\in\Sigma^+}\frak g_{\alpha}.
\]
Let $A=\exp(\frak a)$, $A^+=\exp(\frak a^+)$, and $N=\exp(\frak n)$.
Let $G=KAN$ be the Iwasawa decomposition defined by $K$ and $\frak a^+$. Let $Z(A)$ be the centralizer of $A$ in $G$ and  
 $M=Z(A)\cap K$. The  group $M$ normalizes $N$. Let $P=MAN$ be the minimal parabolic subgroup of $G$ associated to $K$ and $\frak a^+$.
Let $\nu$ be the unique Borel  $K$-invariant probability measure on the Furstenberg-Poisson boundary $G/P$. 
Let 
\[
\pi:G\to \mathcal{U}(L^2(G/P,\nu))	
\]
be the associated quasi-regular representation defined, on any $\varphi\in L^2(G/P,\nu)$,  by the formula
\[
	(\pi(g)\varphi)(b)=\varphi(g^{-1}b)\sqrt{\frac{{\rm d}g_*\nu}{{\rm d}\nu}(b)}.
\]
Let 
\[
	\Xi(g)=\langle\pi(g)\mathbbm 1_{G/P},\mathbbm 1_{G/P}\rangle
\]
be the Harish-Chandra function of $G$. 
Let $d_G$ be the left-invariant Riemannian metric on $G$ whose scalar product on the tangent space of $G$ at the identity $e$ is defined by the Killing form $B$ and the Cartan involution $\Theta$ associated to the decomposition
\[
	\frak g=\frak k\oplus\frak p,
\]
as
\[
	\langle Z,Y\rangle=-B(Z,\Theta(Y))\quad \forall Z,Y\in\frak g.
\]
It is convenient to consider also normalizations of the scalar product $\langle \cdot,\cdot \rangle$ and $d_G$. Once a normalization is chosen,
let $d_X$ be the unique Riemannian metric on the symmetric space $X=G/K$
such that the canonical projection $G\to G/K$ becomes a $G$-equivariant Riemannian submersion.
Let $x_0\in X$ be the image of the identity element of $G$ under the canonical projection.
Notice that if $H\in\frak a$, then
\[
	d_X(\exp(H)x_0,x_0)=\|H\|=\sqrt{\langle H,H\rangle}.
\]
We define the length function $L:G\to[0,\infty[$ as 
$$
L(g)=d_{X}(gx_{0},x_{0}).
$$
It is a proper function on $G$ which is symmetric and $K$-bi-invariant.
If $E$ is a subset of $G$ and $T>0$ we denote
$$E_{T}=\{g\in E : L(g)<T\}.$$
\subsection{Asymptotic Schur orthogonality relations for lattices}
\begin{theo}\label{theo: main} Let $G$ be a non-compact connected semi-simple Lie group with finite center. Let $\Gamma$ be a lattice in $G$. With the notation as above, for any continuous functions $\varphi,\psi,\varphi',\psi'$ on $G/P$,
$$
\lim_{T\to\infty}\frac{1}{|\Gamma_{T}|}\sum_{\gamma\in \Gamma_{T}}\frac{\langle\pi(\gamma)\varphi,\psi\rangle\overline{\langle\pi(\gamma)\varphi',\psi'\rangle}}{\Xi^{2}(\gamma)}=\langle \varphi,\varphi'\rangle\overline{\langle\psi,\psi'\rangle}.
$$
\end{theo}
The above result does not extend to all vectors in $L^{2}(G/P,\nu)$ in general, as the following corollary shows.

\begin{coro}\label{cor: equivalences in rank 1} Let $G$ be a connected semi-simple Lie group with finite center.  Assume  the real rank of $G$ equals $1$. Let $\Gamma$ be a lattice in $G$. The following three conditions are equivalent
\begin{enumerate}
\item For any  functions $\varphi,\psi,\varphi',\psi'$ in $L^{2}(G/P,\nu)$,
$$
\lim_{T\to\infty}\frac{1}{|\Gamma_{T}|}\sum_{\gamma\in \Gamma_{T}}\frac{\langle\pi(\gamma)\varphi,\psi\rangle\overline{\langle\pi(\gamma)\varphi',\psi'\rangle}}{\Xi^{2}(\gamma)}=\langle \varphi,\varphi'\rangle\overline{\langle\psi,\psi'\rangle}.
$$
\item The restricted representation $\pi|_{\Gamma}$ satisfies property RD
with respect to $L|_{\Gamma}$. More precisely, there exists a constant
$C>0$ such that, for all unit vectors
$\varphi,\psi\in L^{2}(G/P,\nu)$ and all $T\geq1$,
\[
\sum_{\gamma\in\Gamma_{T+1}\setminus\Gamma_T}
|\langle\pi(\gamma)\varphi,\psi\rangle|^{2}
\leq CT^{2}.
\]
The exponent $2$ is optimal.

\item The lattice $\Gamma$ is uniform.
\end{enumerate}
\end{coro}

\noindent
The estimate in (2) is a precise annular form of rapid decay. Summing it over unit-width annuli gives a bound of order $(1+T)^3$ for the sum of squared matrix coefficients on $\Gamma_T$, uniformly over unit vectors. The annular exponent $2$ is optimal: for $\varphi=\psi=\mathbbm 1_{G/P}$, Lemma~\ref{lemma: HC} and lattice-point counting give
\[
\sum_{\gamma\in\Gamma_{T+1}\setminus\Gamma_T}\Xi(\gamma)^2
\asymp T^2.
\]
Here $\asymp$ means that the ratio is bounded above and below by positive constants for all sufficiently large $T$. Throughout, rapid decay refers to the geometric length $L|_\Gamma$; for nonuniform lattices in rank one this must not be confused with rapid decay for a word length.

\subsection{Double equidistribution on the boundary}\label{section: double}
The proof uses an equidistribution result on $G/P\times G/P$. Related results for hyperbolic groups appear in \cite[Theorem 3.2]{BG} inspired by the fundamental work of Roblin, see \cite[Chapitre 4]{Rob}. In the semisimple setting, we deduce the version needed here from the bisector-counting theorem of Gorodnik and Oh \cite[Theorem 1.6]{GorOh}, together with the concentration estimates in \cite[Proposition 3.3]{BLP}. A version with a convergence rate is stated in  \cite[Theorem 7.1 and Theorem 7.2]{GorNev}. Its formulation uses the polar decomposition $KA^+K\subset G$ and the map 
\[
	\bold{b}:KA^+K\to G/P
\]
\[
	\bold{b}(k_1ak_2)= k_1P.
\]
Notice that $KA^+K$ is open and dense in $G$. The map $\bold{b}$ is well defined (see \cite[Theorem 5.20 and its proof]{Kna} or \cite[Theorem 1.1 and Corollary 1.2 of Ch. IX and their proofs]{HelDifGeo}).
For $g\in KA^{+}K$ it will be convenient to write 
$$\hat{g}=\bold{b}(g).$$
Notice that if $g\in KA^{+}K$ then $g^{-1}\in KA^{+}K$ (see the proof of Lemma \ref{lemma: half Roblin} below). We write 
$$\check{g}=\bold{b}(g^{-1}).$$
If $E$ is a subset of $G$, we define $E^{+}=E\cap KA^+K$. 

\begin{prop}\label{prop: equidistribution}  Let $G$ be a non-compact connected semisimple Lie group with finite center. Let $\Gamma$ be a lattice in $G$. With the notation as above, for any continuous function $F$ on $G/P\times G/P$,
$$
\lim_{T\to\infty}\frac{1}{|\Gamma^{+}_{T}|}\sum_{\gamma\in \Gamma^{+}_{T}}F\left({\hat \gamma},{\check{\gamma}}\right)=\int_{G/P}\int_{G/P}F(a,b)d\nu(a)d\nu(b).
$$
\end{prop}

For technical reasons (namely to avoid dealing with tangent convergence near the boundary) we will prove that the equidistribution also holds if we retain only lattice points whose $\frak a^{+}$ component is bounded away from the walls enclosing $\frak a^{+}$. In order to make this precise, recall that
the linear form defined as half the sum of the positive roots counted with multiplicity
\[
	\rho=\frac{1}{2}\sum_{\alpha\in\Sigma^+}\dim\frak g_{\alpha}\cdot\alpha
\]
reaches its maximum  on 
$$\{H\in\frak a^+ : \|H\|\leq 1\}$$ 
at a unique vector  $H_{max}$ of norm $1$ called the barycenter of $\frak a^+$.
Let $\theta>0$. We define the open cone in $\frak a^+$ of angle $2\theta$ around $H_{max}$ as
\[
	\frak a^{\theta}=\{H\in\frak a^+: \measuredangle (H, H_{max})<\theta\},
\]
Let $A^{\theta}=\exp(\frak a^{\theta})$.
Notice that if $g\in KA^{\theta}K$ then  $g^{-1}\in KA^{\theta}K$ (see the proof of Lemma \ref{lemma: half Roblin} below and notice that the opposite involution is the restriction of an Euclidean isometry which preserves $\frak{a^{+}}$ and fixes $H_{max}$).
If $E$ is a subset of $G$, we define $E^{\theta}=E\cap KA^{\theta}K$. 

\begin{prop}\label{prop: equidistribution away from the walls}  Let $G$ be a non-compact connected semisimple Lie group with finite center. Let $\Gamma$ be a lattice in $G$. Let $\theta>0$. With the notation as above, for any continuous function $F$ on $G/P\times G/P$,
$$
\lim_{T\to\infty}\frac{1}{|\Gamma^{\theta}_{T}|}\sum_{\gamma\in \Gamma^{\theta}_{T}}F\left({\hat \gamma},{\check{\gamma}}\right)=\int_{G/P}\int_{G/P}F(a,b)d\nu(a)d\nu(b).
$$
\end{prop}
Notice that when $\theta$ is large enough, the condition $\measuredangle (H, H_{max})<\theta$ is void, hence for those $\theta$ we have $\frak a^{\theta}=\frak a^{+}$, hence $A^{\theta}=A^{+}$ and  $\Gamma^{\theta}=\Gamma^{+}$. This shows that Proposition $\ref{prop: equidistribution away from the walls}$ contains  Proposition $\ref{prop: equidistribution}$.  

\section{Proof of the double equidistribution}
In order to prove Proposition \ref{prop: equidistribution away from the walls}, notice that $G/P$ is a differentiable manifold and the inclusion $K\subset G$ induces a diffeomorphism $K/M\to G/P$. 
If $p:K\to K/M$ denotes the canonical projection, the relevant measure  viewed on $K/M$ is 
$\nu=p_{*}\mu$ where $\mu$ is the probability Haar measure on $K$. The Borel probability space $(K/M,\nu)$ is separable and its topology can be defined by a metric. The same is true for  the Borel probability space $(K/M\times K/M,\nu\otimes\nu)$. If $(a,b)\in K/M\times K/M$ and if $D_{(a,b)}$ denotes the corresponding Dirac measure, if we use the notation $\eta=\nu\otimes\nu$ and
$$
\eta_{T}=\frac{1}{|\Gamma^{\theta}_{T}|}\sum_{\gamma\in \Gamma^{\theta}_{T}}D_{\left({\hat \gamma},{\check{\gamma}}\right)},
$$
we may reformulate  Proposition  \ref{prop: equidistribution away from the walls} as the weak convergence of the probability measures $\eta_{T}$ to the measure $\eta$:
$$
\lim_{T\to\infty}\eta_{T}=\eta.
$$
According to \cite[Theorem 2.8, (i)]{Bil}, the above convergence is equivalent to convergence  on sets of the type $U\times V$ where $U,V$ are  sets of continuity for $\nu$. This is the content of the next lemma (which improves \cite[Lemma 3.4]{BLP}):
\begin{lemma}\label{lemma: half Roblin} Let $U,V$ be Borel subsets of $K/M$ with $\nu(\partial U)=\nu(\partial V)=0$.
Let $\theta>0$. Then
\[
	\lim_{T\to\infty}\frac{|\{\gamma\in\Gamma^{\theta}_T: {\hat \gamma}\in U\,\, \mbox{and}\,\, {\check{\gamma}}\in V \}|}{|\Gamma^{\theta}_T|}=\nu(U)\nu(V).
\]	
\end{lemma}
\begin{proof}
Let $\tilde U=p^{-1}(U)$ and $\tilde V=p^{-1}(V)$. 
Let $N(A)$ be the normalizer of $A$ in $G$ and let $M'=N(A)\cap K$. Recall that the Weyl group is the quotient $W=M'/M$ and that it contains a unique element $s_0$ which sends $\frak{a}^{+}$ to $-\frak{a}^{+}$.
Let $m_0\in M'$ such that $m_0M=s_0$. Recall that the \emph{opposition involution}
\[
	\iota:\frak{a}^+\to\frak{a}^+,\quad H\mapsto -\mathrm{Ad}(m_0)H
\]
satisfies $\exp(-H)=m_0^{-1}\exp(\iota H)m_0$  for all $H\in\frak{a}^+$.
Let $g\in KA^{+}K$. It means there exists $k,l\in K$ and $H\in\frak a^{+}$ such that $g=k\exp(H)l$. Using the identity
$$g^{-1}=l^{-1}\exp(-H)k^{-1}=l^{-1}m_0^{-1}\exp(\iota H)m_0k^{-1},$$ 
it is elementary to check the equality
$$
\{g\in KA^{+}K: {\hat g}\in U\,\, \mbox{and}\,\, {\check{g}}\in V \}=\tilde UA^+m_0^{-1}(\tilde V)^{-1}.
$$
Let $T>0$. Using twice the inclusion $\Gamma^{\theta}_{T}\subset KA^{+}K$ and the above equality we obtain:
\begin{align*}
&\{\gamma\in\Gamma^{\theta}_T: {\hat \gamma}\in U\,\, \mbox{and}\,\, {\check{\gamma}}\in V \}\\
&=\Gamma^{\theta}_T\cap\{g\in KA^{+}K : {\hat g}\in U\,\, \mbox{and}\,\, {\check{g}}\in V \}\\
&=\Gamma^{\theta}_T\cap\tilde UA^+m_0^{-1}(\tilde V)^{-1}\\
&=\Gamma^{\theta}_T\cap\tilde U\overline{A^+}m_0^{-1}(\tilde V)^{-1}.
\end{align*}
In order to apply \cite[Theorem 1.6]{GorOh}, we put $\Omega_1=\tilde U$ and $\Omega_2=m_0^{-1}(\tilde V)^{-1}$. On one hand we have:
\begin{align*}
\tilde U\overline{A^+}m_0^{-1}(\tilde V)^{-1}&=\tilde UM\overline{A^+}m_0^{-1}(\tilde V)^{-1}=\tilde U\overline{A^+}Mm_0^{-1}(\tilde V)^{-1}\\
					&=\Omega_1\overline{A^+}M\Omega_2.
\end{align*} 
On the other hand, for $i=1,2$, $\mu(\partial\Omega_i)=0$ because $\nu(\partial U)=\nu(\partial V)=0$. Moreover, $\mu(\Omega_1 M)=\mu(\tilde U M)=\mu(\tilde U)=\nu(U)$, and, as $M$ is normal in $M'$,
\begin{align*}
\mu(M\Omega_2)&=\mu(Mm_0^{-1}(\tilde V)^{-1})=\mu(m_0Mm_0^{-1}(\tilde V)^{-1})\\
              &=\mu(M(\tilde V)^{-1})=\mu(\tilde V M)=\mu(\tilde V)=\nu(V).
\end{align*}
According to \cite[Theorem 1.6]{GorOh},
\[
\lim_{T\to\infty}
\frac{|(\Gamma\cap G_T)\cap\Omega_1\overline{A^+}M\Omega_2|\,\mathrm{vol}(G/\Gamma)}{\mathrm{vol}(G_T)}
=\nu(U)\nu(V).
\]
Notice that
\[
\frac{|(\Gamma\cap G_T)\setminus \Gamma^{\theta}_T|+|\Gamma^{\theta}_T|}{|\Gamma\cap G_T|}=1.
\]
According  to  \cite[Proposition 3.3 and its proof]{BLP}
\[
	\lim_{T\to\infty}\frac{|(\Gamma\cap G_T)\setminus \Gamma^{\theta}_T|}{|\Gamma\cap G_T|}=0.
\]
Finally, according to \cite{EskMcM}
\begin{equation}\label{equation: EskMcM}
	|\Gamma\cap G_T|\cdot \mathrm{vol}(G/\Gamma)\sim \mathrm{vol}(G_T).
\end{equation}
Combining these facts with the preceding normalized limit proves the lemma.
	
\end{proof}

\section[Normalized coefficients near infinity]{Behavior of normalized coefficients of continuous functions near infinity}\label{section: behavior near infinity} Recall that the Howe--Moore theorem \cite{HM} states that every unitary representation of a connected noncompact simple Lie group with finite center and without nonzero invariant vector has all its matrix coefficients vanishing at infinity. In order to study the asymptotic behavior of a coefficient of the quasi-regular representation at infinity, we  normalize it with the coefficient defined by the constant function on the boundary $B=K/M$:
$$
g\mapsto\frac{\langle\pi(g)\varphi,\psi\rangle}{\langle\pi(g)\mathbbm 1_{B},\mathbbm 1_{B}\rangle}=\frac{\langle\pi(g)\varphi,\psi\rangle}{\Xi(g)}.
$$
For $\psi\in L^{2}(B,\nu)$, recall that the normalized square-root of the Poisson kernel 	is defined as
$$
\mathcal{P}_0\psi(g)=\frac{\langle\pi(g)\mathbbm 1_B,\overline{\psi}\rangle}{\Xi(g)}.
$$
Notice that $\mathcal{P}_0\psi$ is $K$-right invariant. Hence it can also  be viewed as a function on $X=G/K$.
For continuous $\psi$, the normalized Poisson transform recovers the boundary value along regular approaches to the boundary. More generally, for continuous $\varphi$ and $\psi$, the difference
\[
\frac{\langle\pi(g)\varphi,\psi\rangle}{\Xi(g)}
-\varphi(\check g)\overline{\psi(\hat g)}
\]
tends uniformly to zero as $L(g)\to\infty$ in a fixed cone bounded away from the walls of a Weyl chamber. The next lemma makes this statement precise. 
\begin{lemma}\label{lemma: behavior of normalized smooth coefficients} With the notation as above, let $G$ be a connected non-compact semi-simple Lie group with finite center. We fix a left $K$-invariant Riemannian metric  $d$ on $B=K/M$. For $b\in B$ and $r>0$, we denote 
$$
B(b,r)=\{a\in B: d(a,b)<r\}
$$
the open ball of center $b$ and radius $r$.
 Let $\varphi,\psi:B\to\mathbb C$ be bounded Borel functions. For any $g\in KA^{+}K$,
we define 
$$
R_{\varphi,\psi}(g)=\frac{\langle\pi(g)\varphi,\psi\rangle}{\Xi(g)}-\varphi\left({\check{g}}\right)\overline{\psi({\hat g})}.
$$ 
The following statements are true.
\begin{enumerate}

\item
\begin{align*}
R_{\varphi,\psi}(g)&=\frac{\langle\pi(g)\varphi,\psi-\psi({\hat g})\mathbbm 1_B\rangle}{\Xi(g)}\\
&+\overline{\psi({\hat g})}\frac{\left\langle\varphi-\varphi\left({\hat {g^{-1}}}\right)\mathbbm 1_B,\pi(g^{-1})\mathbbm 1_B\right\rangle}{\Xi(g)},
\end{align*}

\item
\begin{align*}
\left|R_{\varphi,\psi}(g)\right|&\leq\|\varphi\|_{\infty}\mathcal{P}_0|\psi-\psi({\hat g})\mathbbm 1_B|(g)\\
&+|\psi({\hat g})|\mathcal{P}_0\left|\varphi-\varphi\left({\check{g}}\right)\mathbbm 1_B\right|(g^{-1}).
\end{align*}

\item
\begin{align*}
\mathcal{P}_0|\psi-\psi(\hat{g})\mathbbm 1_B|(g)&\leq\sup_{b\in B(\hat g,r)}|\psi(b)-\psi(\hat{g})|\\
&+\sup_{b\in B}|\psi(b)-\psi(\hat{g})|\mathcal{P}_0\mathbbm 1_{B\setminus B(\hat g,r)}(g).
\end{align*}

\item Assume $\theta>0$ is small enough so that $\overline{\frak a^{\theta}}\subset {\frak a}^{+}\cup\{0\}$. If $\psi$ is continuous  then
$$
\lim_{T\to\infty}\sup_{g\in (KA^{\theta}K)\setminus G_{T}}\mathcal{P}_0|\psi-\psi(\hat{g})\mathbbm 1_B|(g)=0.
$$

\item Assume $\theta>0$ is small enough so that $\overline{\frak a^{\theta}}\subset {\frak a}^{+}\cup\{0\}$. If $\varphi$ and $\psi$ are continuous  then
$$
\lim_{T\to\infty}\sup_{g\in (KA^{\theta}K)\setminus G_{T}}\left|R_{\varphi,\psi}(g)\right|=0.
$$
\end{enumerate}
\end{lemma}

\begin{proof} In order to prove the first statement we compute:
\begin{align*}
\frac{\langle\pi(g)\varphi,\psi\rangle}{\Xi(g)}-\varphi\left({\check{g}}\right)
\overline{\psi({\hat g})}&=\frac{\langle\pi(g)\varphi,\psi-\psi({\hat g})\mathbbm 1_B+\psi({\hat g})\mathbbm 1_B\rangle}{\Xi(g)}\\
&\,\,\,\,\,\,\,\,-\varphi\left({\check{g}}\right)\overline{\psi({\hat g})}\\
&=\frac{\langle\pi(g)\varphi,\psi-\psi({\hat g})\mathbbm 1_B\rangle}{\Xi(g)}+\overline{\psi({\hat g})}\frac{\langle\pi(g)\varphi,\mathbbm 1_B\rangle}{\Xi(g)}\\
&\,\,\,\,\,\,\,\,-\varphi\left({\check{g}}\right)\overline{\psi({\hat g})}\\
&=\frac{\langle\pi(g)\varphi,\psi-\psi({\hat g})\mathbbm 1_B\rangle}{\Xi(g)}+\overline{\psi({\hat g})}\frac{\langle\varphi,\pi(g^{-1})\mathbbm 1_B\rangle}{\Xi(g)}\\
&\,\,\,\,\,\,\,\,-\varphi\left({\check{g}}\right)\overline{\psi({\hat g})}\frac{\langle\mathbbm 1_B,\pi(g^{-1})\mathbbm 1_B\rangle}{\Xi(g)}\\
&=\frac{\langle\pi(g)\varphi,\psi-\psi({\hat g})\mathbbm 1_B\rangle}{\Xi(g)}\\
&\,\,\,\,\,\,\,\,+\overline{\psi({\hat g})}\frac{\left\langle\varphi-\varphi\left({\check{g}}\right)\mathbbm 1_B,\pi(g^{-1})\mathbbm 1_B\right\rangle}{\Xi(g)}.
\end{align*}
Taking norms proves the next statement:
\begin{align*}
\left|\frac{\langle\pi(g)\varphi,\psi\rangle}{\Xi(g)}-\varphi\left({\check{g}}\right)
\overline{\psi({\hat g})}\right|&\leq\|\varphi\|_{\infty}\frac{\langle\pi(g)\mathbbm 1_B,|\psi-\psi({\hat g})\mathbbm 1_B|\rangle}{\Xi(g)}\\
&\,\,\,\,\,\,\,\,+|\psi({\hat g})|\frac{\langle|\varphi-\varphi\left({\check{g}}\right)\mathbbm 1_B|,\pi(g^{-1})\mathbbm 1_B\rangle}{\Xi(g)}\\
&=\|\varphi\|_{\infty}\mathcal{P}_0|\psi-\psi({\hat g})\mathbbm 1_B|(g)\\
&\,\,\,\,\,\,\,\,+|\psi({\hat g})|\mathcal{P}_0|\varphi-\varphi\left({\check{g}}\right)\mathbbm 1_B|(g^{-1}).
\end{align*}
In order to prove the third statement we define
$$
K(g,b)=\frac{\sqrt{\frac{dg_{*}\nu}{d\nu}(b)}}{\Xi(g)}.
$$
Notice that $K(g,b)>0$ and
$\int_{B}K(g,b)d\nu(b)=1$.
We have
\begin{align*}
\mathcal{P}_0|\psi-\psi(\hat{g})\mathbbm 1_B|(g)&=\int_{B}K(g,b)|\psi(b)-\psi(\hat{g})|d\nu(b)\\
&=\int_{B(\hat{g},r)}K(g,b)|\psi(b)-\psi(\hat{g})|d\nu(b)\\
&+\int_{B\setminus B(\hat g,r)}K(g,b)|\psi(b)-\psi(\hat{g})|d\nu(b)\\
&\leq\sup_{b\in B(\hat g,r)}|\psi(b)-\psi(\hat{g})|\\
&+\sup_{b\in B}|\psi(b)-\psi(\hat{g})|\mathcal{P}_0\mathbbm 1_{B\setminus B(\hat g,r)}(g).
\end{align*}
To prove the next statement, notice first that according to \cite[Lemma 3.2]{BLP}, for each $r>0$ there exists a function $h_r:[0,\infty[\to[0,\infty[$ with the following properties:
\begin{enumerate}
	\item $h_r$ is non-increasing,
	\item $\lim_{s\to\infty}h_r(s)=0,$
	\item $\forall g\in KA^{\theta}K,$
\[
	\mathcal{P}_0\mathbbm 1_{\{b\in B :\, d(b,\hat{g})\geq r\}}(g)\leq h_r(L(g)).
\]
\end{enumerate}
where $d$ and $\theta$ are as specified above. If $\epsilon>0$ is given, as $\psi$ is uniformly continuous on the compact space $B$, there exists $r>0$ small enough so that for any $g\in KA^{+}K$,
$$
\sup_{b\in B(\hat g,r)}|\psi(b)-\psi(\hat{g})|<\epsilon/2.
$$
Choose $T>0$ sufficiently large that 
$$
2\|\psi\|_\infty h_r(T)<\epsilon/2.
$$
 Applying the previously proved inequality stated in  the lemma, we deduce that
for all $T'\geq T$, for all $g\in(KA^{\theta}K)\setminus G_{T'}$
$$
\mathcal{P}_0|\psi-\psi(\hat{g})\mathbbm 1_B|(g)<\epsilon.
$$
As already mentioned, the opposition involution preserves $\frak a^\theta$, moreover, since $L(g)=L(g^{-1})$, statements (2) and (4), applied to $g$ and $g^{-1}$, prove (5).
\end{proof}

\section{Proof of the asymptotic Schur orthogonality relations}
We start with an elementary and useful lemma.

\begin{lemma}\label{lemma: balls and coronae} Let $X$ be a set and $(B_{T})_{T>0}$ be an increasing family of subset of $X$; that is, if $S<T$, then $B_{S}\subset B_{T}\subset X$. We assume each set $B_{T}$ is finite (and non-empty). Suppose
$$
\lim_{T\to\infty}\frac{|B_{T}|}{|B_{T+1}|}=c<1.
$$ 
Let $f:X\to\mathbb C$ be a function.
Then
$$
\lim_{T\to\infty}\frac{1}{|B_{T}|}\sum_{x\in B_{T}}f(x)=l,
$$ 
if and only if
$$
\lim_{T\to\infty}\frac{1}{|B_{T+1}\setminus B_{T}|}\sum_{x\in B_{T+1}\setminus B_{T}}f(x)=l.
$$ 
\end{lemma}
\begin{proof} We have:
\begin{align*}
\frac{1}{|B_{T+1}|}\sum_{x\in B_{T+1}}f(x)&=\frac{|B_{T}|}{|B_{T+1}|}\left(\frac{1}{|B_{T}|}\sum_{x\in B_{T}}f(x)\right)\\
&+\frac{|B_{T+1}\setminus B_{T}|}{|B_{T+1}|}\left(\frac{1}{|B_{T+1}\setminus B_{T}|}\sum_{x\in B_{T+1}\setminus B_{T}}f(x)\right).
\end{align*}
As $1-c>0$ this proves the lemma.
\end{proof}

\subsection{Proof of Theorem \ref{theo: main}}
We denote $r=\mathrm{rank}_{\mathbb R}G$ the real rank of $G$. We assume $r\geq 1$.
The exponential growth rate of the Riemannian symmetric space $(X,d_X)$, also called the volume entropy, 
$$
\delta=\lim_{T\to\infty}\frac{\log\mbox{Vol}(B(x_{0},T))}{T}
$$
satisfies
 $$\delta=2\rho(H_{max})>0.$$
To each lattice  $\Gamma$ in $G$ corresponds a constant $c>0$, such that for any $\theta>0$, 
$$
|\Gamma^{\theta}_{T}|\sim c T^{(r-1)/2}\cdot e^{\delta T}.
$$
See \cite[Sections 3.2 and 3.5]{BLP}.

In particular,
$$
\lim_{T\to\infty}\frac{|\Gamma^{\theta}_{T}|}{|\Gamma^{\theta}_{T+1}|}=1/e^{\delta}<1.
$$
Our first goal  is to prove the following key lemma.

\begin{lemma}\label{lemma: Schur on theta-coronae} With the notation as above, let $\theta>0$ be small enough so that $\overline{\frak a^{\theta}}\subset {\frak a}^{+}\cup\{0\}$. For any continuous functions $\varphi,\varphi',\psi,\psi'$ on $G/P$,
$$
\lim_{T\to\infty}\frac{1}{|\Gamma^{\theta}_{T+1}\setminus\Gamma^{\theta}_{T}|}\sum_{\gamma\in \Gamma^{\theta}_{T+1}\setminus\Gamma^{\theta}_{T}}\frac{\langle\pi(\gamma)\varphi,\psi\rangle\overline{\langle\pi(\gamma)\varphi',\psi'\rangle}}{\Xi^{2}(\gamma)}=\langle \varphi,\varphi'\rangle\overline{\langle\psi,\psi'\rangle}.
$$
\end{lemma}

\begin{proof}
Let  $\varphi,\psi,\varphi',\psi'$ be continuous functions on $G/P$. Recall that for any $g\in KA^{+}K$,
$$
\frac{\langle\pi(g)\varphi,\psi\rangle}{\Xi(g)}=\varphi\left({\check{g}}\right)\overline{\psi({\hat g})}+R_{\varphi,\psi}(g).
$$
We define the function
$$
F:G/P\times G/P\to\mathbb C
$$
$$
(a,b)\mapsto\psi'(a)\overline{\psi(a)}\varphi(b)\overline{\varphi'(b)},
$$
and we obtain, for any $g\in KA^{+}K$,
\begin{align*}
\frac{\langle\pi(g)\varphi,\psi\rangle\overline{\langle\pi(g)\varphi',\psi'\rangle}}{\Xi^{2}(g)}&=F(\hat{g},\check{g})\\
&+R_{\varphi,\psi}(g)\overline{\varphi'\left({\check{g}}\right)}\psi'({\hat g})+\overline{R_{\varphi',\psi'}(g)}\varphi\left({\check{g}}\right)\overline{\psi({\hat g})}\\
&+R_{\varphi,\psi}(g)\overline{R_{\varphi',\psi'}(g)}.
\end{align*}
As the function $F$ is continuous, we may apply Proposition \ref{prop: equidistribution away from the walls}  and Lemma \ref{lemma: balls and coronae} to deduce that
$$
\lim_{T\to\infty}\frac{1}{|\Gamma^{\theta}_{T+1}\setminus\Gamma^{\theta}_{T}|}\sum_{\gamma\in \Gamma^{\theta}_{T+1}\setminus\Gamma^{\theta}_{T}}F(\hat{\gamma},\check{\gamma})=\langle \varphi,\varphi'\rangle\overline{\langle\psi,\psi'\rangle}.
$$ 
As the functions $\varphi,\psi,\varphi',\psi'$ are continuous on the compact space $G/P$ they are bounded. Hence, according to point (5) of Lemma \ref{lemma: behavior of normalized smooth coefficients} (which applies because of the hypothesis on $\theta$) the remaining means
\begin{align*}
&\frac{1}{|\Gamma^{\theta}_{T+1}\setminus\Gamma^{\theta}_{T}|}\sum_{\gamma\in \Gamma^{\theta}_{T+1}\setminus\Gamma^{\theta}_{T}}\left(R_{\varphi,\psi}(\gamma)\overline{\varphi'\left({\check{\gamma}}\right)}\psi'({\hat \gamma})+\overline{R_{\varphi',\psi'}(\gamma)}\varphi\left({\check{\gamma}}\right)\overline{\psi({\hat \gamma})}\right)\\
&+\frac{1}{|\Gamma^{\theta}_{T+1}\setminus\Gamma^{\theta}_{T}|}\sum_{\gamma\in \Gamma^{\theta}_{T+1}\setminus\Gamma^{\theta}_{T}}R_{\varphi,\psi}(\gamma)\overline{R_{\varphi',\psi'}(\gamma)}
\end{align*}
tend to zero if $T$ tends to infinity. 
\end{proof}

The following statement is a technical and more general version of the main theorem. 

\begin{theo}\label{theo: main technical}  Let $G$ be a non-compact connected semi-simple Lie group with finite center. Let $\Gamma$ be a lattice in $G$. Suppose $E$ is a subset of $\Gamma$ and there exists $\theta>0$ such that
$$\Gamma^{\theta}\subset E\subset\Gamma.$$
Then, for any continuous functions $\varphi,\psi,\varphi',\psi'$ on $G/P$,
$$
\lim_{T\to\infty}\frac{1}{|E_{T}|}\sum_{\gamma\in E_{T}}\frac{\langle\pi(\gamma)\varphi,\psi\rangle\overline{\langle\pi(\gamma)\varphi',\psi'\rangle}}{\Xi^{2}(\gamma)}=\langle \varphi,\varphi'\rangle\overline{\langle\psi,\psi'\rangle},
$$
and 
$$
\lim_{T\to\infty}\frac{1}{|E_{T+1}\setminus E_{T}|}\sum_{\gamma\in E_{T+1}\setminus E_{T}}\frac{\langle\pi(\gamma)\varphi,\psi\rangle\overline{\langle\pi(\gamma)\varphi',\psi'\rangle}}{\Xi^{2}(\gamma)}=\langle \varphi,\varphi'\rangle\overline{\langle\psi,\psi'\rangle},
$$
\end{theo}
\begin{proof} As the correspondance $\theta\mapsto\Gamma^{\theta}$ is increasing for the inclusion, we may assume $\theta>0$ is small enough so that  $\overline{\frak a^{\theta}}\subset {\frak a}^{+}\cup\{0\}$. We have:
\begin{align*}
&\frac{1}{|E_{T}|}\sum_{\gamma\in E_{T}}\frac{\langle\pi(\gamma)\varphi,\psi\rangle\overline{\langle\pi(\gamma)\varphi',\psi'\rangle}}{\Xi^{2}(\gamma)}\\
&=\frac{|\Gamma^{\theta}_{T}|}{|E_{T}|}\left(\frac{1}{|\Gamma^{\theta}_{T}|}\sum_{\gamma\in \Gamma^{\theta}_{T}}\frac{\langle\pi(\gamma)\varphi,\psi\rangle\overline{\langle\pi(\gamma)\varphi',\psi'\rangle}}{\Xi^{2}(\gamma)}\right)\\
&+\frac{1}{|E_T|}\sum_{\gamma\in E_T\setminus\Gamma_T^\theta}
\frac{\langle\pi(\gamma)\varphi,\psi\rangle\overline{\langle\pi(\gamma)\varphi',\psi'\rangle}}{\Xi^2(\gamma)}.
\end{align*}		
On one hand, there is a constant $c>0$ such that, for any $\theta>0$, 
$$
|\Gamma^{\theta}_{T}|\sim|\Gamma_{T}|\sim c T^{(r-1)/2}\cdot e^{\delta T},
$$
on the other hand, positivity of the kernel defining $\pi$ gives
\[
|\langle\pi(g)\varphi,\psi\rangle|
\leq\|\varphi\|_\infty\|\psi\|_\infty\Xi(g).
\]
Consequently, the exceptional contribution satisfies
\begin{align*}
&\left|\frac{1}{|E_T|}\sum_{\gamma\in E_T\setminus\Gamma_T^\theta}
\frac{\langle\pi(\gamma)\varphi,\psi\rangle
\overline{\langle\pi(\gamma)\varphi',\psi'\rangle}}{\Xi(\gamma)^2}\right|\\
&\qquad\leq\frac{|E_T\setminus\Gamma_T^\theta|}{|E_T|}
\|\varphi\|_\infty\|\psi\|_\infty
\|\varphi'\|_\infty\|\psi'\|_\infty\longrightarrow0.
\end{align*}
This estimate remains valid when $E_T\setminus\Gamma_T^\theta$ is empty.
Hence, taking limits, and applying Lemma \ref{lemma: Schur on theta-coronae} and Lemma \ref{lemma: balls and coronae}, we obtain:
\begin{align*}
&\lim_{T\to\infty}\frac{1}{|E_{T}|}\sum_{\gamma\in E_{T}}\frac{\langle\pi(\gamma)\varphi,\psi\rangle\overline{\langle\pi(\gamma)\varphi',\psi'\rangle}}{\Xi^{2}(\gamma)}\\
&=\lim_{T\to\infty}\frac{1}{|\Gamma^{\theta}_{T}|}\sum_{\gamma\in \Gamma^{\theta}_{T}}\frac{\langle\pi(\gamma)\varphi,\psi\rangle\overline{\langle\pi(\gamma)\varphi',\psi'\rangle}}{\Xi^{2}(\gamma)}\\
&=\langle \varphi,\varphi'\rangle\overline{\langle\psi,\psi'\rangle}.
\end{align*}
This proves the first 	statement of the theorem. The second statement follows from the first, thanks to Lemma \ref{lemma: balls and coronae}.
\end{proof}

\subsection{Proof of Corollary \ref{cor: equivalences in rank 1}} We start with a well-known estimate of the Harish-Chandra function expressed in terms of the length function $L$ defined above and the volume entropy $\delta$ of the symmetric space $(X,d_{X})$.

\begin{lemma}\label{lemma: HC}(Asymptotic behavior of the Harish-Chandra function in rank $1$.) Let $G$ be a connected semi-simple Lie group with finite center. Assume $\mbox{rank}_{\mathbb R}G=1$. There is a constant $c>0$ such that for any $g\in G$
$$
\Xi(g)\sim c\cdot L(g)e^{-\frac{\delta}{2}L(g)}
$$
when $L(g)$ goes to infinity.
\end{lemma}

\begin{proof}
Choose the unit vector $H_{max}\in\frak a^+$ and let $\alpha$ be the unique positive indivisible root. By \cite[page 373]{Ank}, there is $c_2>0$ such that
\[
\Xi(\exp(tH_{max}))\sim c_2\,t\alpha(H_{max})e^{-t\rho(H_{max})}
\qquad(t\to\infty).
\]
The Cartan decomposition reads $G=K\overline{\exp\left(\frak a^+\right)}K$. Thus every $g\in G$ can be written
\[
g=k_1\exp(tH_{max})k_2,\qquad k_1,k_2\in K,\quad t\geq0.
\]
Since $k_2$ fixes $x_0$ and $k_1$ acts isometrically and fixes $x_0$, we have
\[
L(g)=d_X(\exp(tH_{max})x_0,x_0)=t.
\]
The $K$-bi-invariance of $\Xi$ gives $\Xi(g)=\Xi(\exp(tH_{max}))$. Finally, $\delta=2\rho(H_{max})$, so the assertion follows with $c=c_2\alpha(H_{max})>0$.
\end{proof}

We first prove that $(1)\implies (2)$. Remember that
$$
\lim_{T\to\infty}\frac{|\Gamma_{T}|}{|\Gamma_{T+1}|}=1/e^{\delta}<1.
$$
Let $T_{0}>0$ big enough so that if  $T>T_{0}$ then $|\Gamma_{T+1}\setminus\Gamma_{T}|\neq 0$. 
Applying Lemma \ref{lemma: balls and coronae}  to hypothesis (1), we deduce that for  any  functions $\varphi,\psi,\varphi',\psi'$ in $L^{2}(G/P,\nu)$,
$$
\lim_{T\to\infty}\frac{1}{|\Gamma_{T+1}\setminus\Gamma_{T}|}\sum_{\gamma\in\Gamma_{T+1}\setminus\Gamma_{T}}\frac{\langle\pi(\gamma)\varphi,\psi\rangle\overline{\langle\pi(\gamma)\varphi',\psi'\rangle}}{\Xi^{2}(\gamma)}=\langle \varphi,\varphi'\rangle\overline{\langle\psi,\psi'\rangle}.
$$
We deduce that for $T>T_{0}$, the sesqui-multilinear forms
$$
M_{T}(\varphi,\psi,\varphi',\psi')=\frac{1}{|\Gamma_{T+1}\setminus\Gamma_{T}|}\sum_{\gamma\in\Gamma_{T+1}\setminus\Gamma_{T}}\frac{\langle\pi(\gamma)\varphi,\psi\rangle\overline{\langle\pi(\gamma)\varphi',\psi'\rangle}}{\Xi^{2}(\gamma)},
$$
form a pointwise bounded family: for any $\varphi,\psi,\varphi',\psi'$ in $L^{2}(G/P,\nu)$,
$$
\sup\{|M_{T}(\varphi,\psi,\varphi',\psi')|:T>T_{0}\}<\infty.
$$
Applying the Banach--Steinhaus theorem successively in each of the four variables (using the conjugate Hilbert space for the conjugate-linear variables), we obtain a uniform bound for the norms of these sesqui-multilinear forms, where
\[
\|M_T\|=\sup_{\max{\{\|\varphi\|_2;\|\psi\|_2;\|\varphi'\|_2;\|\psi'\|_2\}}\leq1}
|M_T(\varphi,\psi,\varphi',\psi')|.
\]
Thus there exists $C_{1}>0$ such that
$$
\sup\{\|M_{T}\|:T>T_{0}\}<C_{1}.
$$
In particular, for any $T>T_{0}$, for any $\varphi,\psi$ such that $\|\varphi\|_{2}=\|\psi\|_{2}=1$,
$$
\sum_{\gamma\in\Gamma_{T+1}\setminus\Gamma_{T}}\frac{|\langle\pi(\gamma)\varphi,\psi\rangle|^{2}}{\Xi^{2}(\gamma)}<C_{1}\cdot|\Gamma_{T+1}\setminus\Gamma_{T}|.
$$
It is then elementary to deduce from Lemma \ref{lemma: HC} and the asymptotic of $|\Gamma_{T}|$ the existence of a constant $C>C_{1}$ such that for all $T\geq 1$,
$$
\sum_{\gamma\in \Gamma_{T+1}\setminus\Gamma_{T}}|\langle\pi(\gamma)\varphi,\psi\rangle|^{2}\leq C\cdot T^{2}.
$$
We now prove the converse $(2)\implies (1)$. According to Theorem \ref{theo: main}, we know that the sequence converges to the announced limit for all continuous functions. As continuous functions form a dense subspace of the Hilbert space $L^2(G/P,\nu)$, it is enough to show that there is a constant $C_{1}>0$ such that
$$
\sup\{\|M_{T}\|:T>T_{0}\}<C_{1}.
$$
Suppose $T>T_{0}$ and  $\varphi,\psi,\varphi',\psi'$ in $L^{2}(G/P,\nu)$ are of Hilbert norm equal to $1$. Applying Cauchy-Schwarz we obtain:

\begin{align*}
&\left|M_T(\varphi,\psi,\varphi',\psi')\right|\\
&=\frac{1}{|\Gamma_{T+1}\setminus\Gamma_T|}
\left|\sum_{\gamma\in\Gamma_{T+1}\setminus\Gamma_T}
\frac{\langle\pi(\gamma)\varphi,\psi\rangle\overline{\langle\pi(\gamma)\varphi',\psi'\rangle}}{\Xi^2(\gamma)}\right|\\
&\leq\frac{1}{|\Gamma_{T+1}\setminus\Gamma_{T}|}\left(\sum_{\gamma\in\Gamma_{T+1}\setminus\Gamma_{T}}\frac{|\langle\pi(\gamma)\varphi,\psi\rangle|^{2}}{\Xi^{2}(\gamma)}\right)^{\frac{1}{2}}\\
&\qquad\times\left(\sum_{\gamma\in \Gamma_{T+1}\setminus\Gamma_{T}}\frac{|\langle\pi(\gamma)\varphi',\psi'\rangle|^{2}}{\Xi^{2}(\gamma)}\right)^{\frac{1}{2}}.
\end{align*}
Applying the hypothesis, it is elementary to deduce from Lemma \ref{lemma: HC} and the asymptotic of $|\Gamma_{T}|$ the existence of a constant $C_{1}>0$ such that for all $T\geq\max(T_{0};1)$,
$$
\left|M_{T}(\varphi,\psi,\varphi',\psi')\right|\leq C_{1}.
$$
By density and the uniform bound, the annular averages therefore converge for arbitrary $L^2$ vectors. Lemma~\ref{lemma: balls and coronae} then yields the asserted convergence of the ball averages.

The implication $(2)\implies (3)$ is deduced from the  combination of the following known results. 
\begin{enumerate}[label=\roman*)]
\item If $\pi|_\Gamma$ has property RD with respect to $L|_\Gamma$, then so does the regular representation of $\Gamma$. This follows from the positive-vector comparison of \cite[Lemma 4.2]{BoyRD}, since $\mathbbm 1_{G/P}$ is a nonzero positive vector. 
\item If the regular representation  of a finitely generated group $E$ has property $RD$ with respect to a length function $L_{E}$ then the regular representation  of any subgroup $H$ of $E$  has property $RD$ with respect to the length function $L_{E}|_{H}$; see \cite[Remark 4.3(1)]{Cha}. 
\item If the regular representation  of a finitely generated amenable group $H$ has property $RD$ with respect to a length function $L_{H}$  then $(H,L_{H})$ has polynomial growth; see the proof of \cite[Theorem 2.8]{Cha}.
\item Every nonuniform lattice contains an infinite cyclic
subgroup of exponential growth for the induced geometric length.
In rank one, see \cite[Section 1.4 and the proof of Corollary 1.2]{BoyRD} and notice 
that Lemma~\ref{lemma: L_{G} and d_{G}} identifies
the length used there with $L$, up to an additive constant.
\end{enumerate}
Condition (2) implies RD by summing over annuli. If $\Gamma$ were nonuniform, the preceding facts would force its exponentially growing cyclic subgroup to have polynomial growth, a contradiction.

The following lemma will be helpful when proving $(3)\implies (2)$ .
\begin{lemma}\label{lemma: L_{G} and d_{G}} Let $G$ be a non compact connected semi-simple Lie group with finite center. Using the notation introduced above, let $d_{G}$ be the left-invariant Riemannian metric on $G$, defined with the help of the Killing form and the chosen Cartan decomposition, and let $L$ be the length function defined with the help of the symmetric space $X=G/K$ associated to $G$. For all $g\in G$,
$$-\emph{diameter}(K)+d_G(e,g)\leq L(g)\leq d_{G}(e,g).$$
\end{lemma}
\begin{proof} The canonical projection
$$p:G\to G/K=X$$
defines a right-$K$ principal bundle. The map $p$ is $G$-equivariant and is a Riemannian submersion relative to the Riemannian structures  associated with $d_{G}$ and $d_{X}$ (recall that $d_{G}$ is $G$-invariant on the left and $K$-invariant on the right). A geodesic in $G$ from $e$ to $g$
projects under $p$ to a path of no greater length  in  $X$ between $x_{0}$ and $gx_{0}$. This proves that $L(g)\leq d_G(e,g)$. Suppose $c$ is an arc length parametrized geodesic path in $X$ between $x_{0}$ and $gx_{0}$. The horizontal subspaces of the Riemannian submersion $p$ define a connection on the principal right $K$-bundle $G$. Let $\tilde{c}$ be the unique horizontal lift of $c$ starting at $e$. It is parametrized by its arc length and has the same length as $c$. The end point $\tilde{c}(L(g))$ differs from $g$ by an element of $K$, that is there exists $k\in K$ such that $\tilde{c}(L(g))k=g$, hence: 
$$d_{G}(e,g)\leq \mbox{length}(\tilde{c})+d_{G}(\tilde{c}(L(g)),\tilde{c}(L(g))k)=L(g)+d_{G}(e,k).$$
This proves that  
$$-\mbox{diameter}(K)+d_G(e,g)\leq L(g).$$
\end{proof}
We are ready to prove $(3)\implies (2)$. The restriction $d_{G}|_{\Gamma}$ of  the Riemannian metric $d_{G}$ on $G$ to the uniform lattice $\Gamma$ belongs to the class $\mathcal D(\Gamma)$ from \cite{BG}. This means that  $d_{G}|_{\Gamma}$ is $\Gamma$-invariant on the left, quasi-isometric to a word metric defined by a finite generating set of $\Gamma$, and Gromov hyperbolic. In rank one, the cocompact lattice  $\Gamma$ is Gromov-hyperbolic and its Gromov boundary may be identified with $G/P$. Its boundary action is amenable by \cite{Ad}, and the associated quasi-regular representation is weakly contained in the regular representation by \cite{Kun}; see also the discussion following \cite[Proposition 2.4]{BG}. Hence we are in the setting of \cite[Proposition 2.4]{BG}. Let us denote $d(K)=\mbox{diameter}(K)$. For each $T>0$ we consider the set
$$
A_{T,d(K)+1}=\{\gamma\in\Gamma: T-(d(K)+1)\leq d_G(e,\gamma)\leq T+d(K)+1\}.
$$ 
According to \cite[Proposition 2.4]{BG} (see also \cite[Theorem 1.2]{BoySph}), there exists a constant $C>0$ such that  for all $T\geq 1$, for all $\varphi,\psi\in L^2(G/P,\nu)$ of Hilbert norm one,
$$
\sum_{\gamma\in A_{T,d(K)+1}}|\langle\pi(\gamma)\varphi,\psi\rangle|^{2}\leq C\cdot T^{2}.
$$
The inclusion  $\Gamma_{T+1}\setminus\Gamma_{T}\subset A_{T,d(K)+1}$ follows from Lemma \ref{lemma: L_{G} and d_{G}}.
This finishes the proof of the corollary.

\end{document}